\documentclass[11pt]{amsart}

\usepackage{amsmath, amsthm, amssymb}
\usepackage{mathrsfs}
\usepackage{url}
\usepackage{amssymb,amsfonts}
\usepackage{color}
\usepackage{shuffle}
\usepackage[usenames,dvipsnames,table]{xcolor}
\usepackage{graphicx}
\usepackage[font={small,it}]{caption}
\usepackage[flushleft]{threeparttable}
\usepackage{longtable}
\usepackage{enumitem}

\usepackage{silence}
\usepackage{mdframed}

\usepackage[top=3cm, bottom=3cm, left=2.5cm, right=3cm]{geometry}

\newtheorem{theorem}{Theorem}
\theoremstyle{plain}

\newtheorem{lemma}[theorem]{Lemma}

\newtheorem{remark}[theorem]{Remark}

\theoremstyle{definition} 

\newtheorem{definition}[theorem]{Definition}
\newtheorem*{tata}{Generalization}
  {\begin{mdframed}[backgroundcolor=lightgray]\begin{tata}}%
  {\end{tata}\end{mdframed}}

\DeclareMathOperator*{\esssup}{ess\,sup}
\DeclareMathOperator*{\essinf}{ess\,inf}

\newcommand{\R}{\mathbb{R}}

\newcommand{\E}{\mathbb{E}}
\newcommand{\F}{\mathcal{F}}

\renewcommand{\P}{\mathbb{P}}

\newcommand{\rp}{\eta}

\newcommand{\Tr}{\operatorname{Tr}}

\begin{document}

\newcommand{\bsdenorm}[1]{\|#1\|_{p,2}} 
\newcommand{\bsdespace}{\mathcal{B}_p}
\newcommand{\bsdenormz}[1]{\|#1\|_{BMO}} 
\newcommand{\bsdespacez}{\mathsf{BMO}}
\newcommand{\drp}{e} 
\newcommand{\dbm}{d} 
\newcommand{\dode}{m} 
\newcommand{\pvar}{{\operatorname{p-var}}}
\newcommand{\qvar}{{\operatorname{q-var}}}
\newcommand{\onevar}{{\operatorname{1-var}}}

\title{Backward stochastic differential equations with Young drift}



\author{Joscha Diehl and Jianfeng Zhang}
\address{Max-Planck Institute for Mathematics in the Sciences, Leipzig.}
\email{diehl@mis.mpg.de}

\address{University of Southern California, Department of Mathematics.}
\email{jianfenz@usc.edu}
\thanks{
  This research was partially supported in part by the DAAD P.R.I.M.E. program
  and NSF grant DMS 1413717.
  Part of this work was carried out while the first author was visiting the University of Souther California,
  and he would like to thank Jin Ma and Jianfeng Zhang for their hospitality.
}

\maketitle

\begin{abstract}
  We prove via a direct fixpoint argument the well-posedness of
  backward stochastic differential equations containing an additional
  drift driven by a path of finite $p$-variation with $p \in [1,2)$.
  An application to the Feynman-Kac representation of semilinear rough partial differential equations is given.
\end{abstract}

\section{Introduction}

Stochastic differential equations (SDEs) driven by Brownian motion $W$ and
an additional deterministic path $\rp$ of low regularity (so called ``mixed SDEs'')
have been well-studied.
In \cite{bib:guerraNualart} the wellposedness of such SDEs is established
if $\rp$ has finite $q$-variation with $q \in [1,2)$.%
\footnote{See Section \ref{sec:youngIntegration} for background on the variation norm and Young integration.}
The integral with respect to the latter is handled via fractional calculus.
Independently, in \cite{bib:diehlPhd} the same problem is studied using
Young integration for the integral with respect to $\rp$.
Interestingly, both approaches need to establish (unique) existence of solutions
via a Yamada-Watanabe theorem. A direct proof using a contraction argument is
not obvious to implement.

For paths of $q$-variation with $q \in (2,3)$ integration has to be dealt with
via the theory of rough paths.
Motivated by a problem in stochastic filtering,
\cite{bib:crisanDiehlFrizOberhauser} give a formal meaning to the mixed SDE by
using a flow decomposition which seperates the stochastic integration from the deterministic
rough path integration. It is not shown that the resulting object actually satisfies
any integral equation.

In \cite{bib:diehlOberhauserRiedel} well-posedness of the corresponding mixed SDE
is established by first constructing a joint rough path ``above'' $W$ and $\rp$.
The determinstic theory of rough paths then allows to solve the mixed SDE.
The main difficulty in that work is the proof of exponential integrability of
the resulting process, which is needed for applications.
In \cite{bib:diehlFrizStannat} these results have been used to study
linear ``rough'' partial differential equations via Feyman-Kac formulae.

Backward stochastic differential equations (BSDEs) were introduced
by Bismut in 1973. In \cite{bib:bismut73} he applied linear BSDEs to stochastic optimal control.
In 1990 Pardoux and Peng \cite{bib:pardouxPeng} then considered non-linear equations.
A solution to a BSDE with driver $f$ and random variable $\xi \in L^2(\F_T)$ is an adapted pair of processes $(Y, Z)$ in suitable spaces, satisfying
\[
  Y_t = \xi + \int_t^T f(r, Y_r, Z_r) ds - \int_t^T Z_r dW_r, \quad t\le T.
\]
Under appropriate conditions on $f$ and $\xi$ they showed the existence of a unique
solution to such an equation.
One important use for BSDEs is their application to semilinear partial differential equations.
This ``nonlinear Feynman-Kac'' formula is for example studied in \cite{bib:pardouxPengPDEs}.

In this work we are interested in showing wellposedness of the following equation
\begin{align}
  \label{eq:introYoungBSDE}
  Y_t
  =
  \xi + \int_t^T f(r,Y_r,Z_r) dr + \int_t^T g(Y_r) d\rp_r - \int_t^T Z_r dW_r.
\end{align}
Here $W$ is a multidimensional Brownian motion,
$\rp$ is a multidimensional (determinstic) path of finite $q$-variation, $q \in [1,2)$
and $\xi$ is a bounded random variable, measurable at time $T$.

Such equations have previously been studied in \cite{bib:diehlFriz}.
In that work $\rp$ is even allowed to be a rough path, i.e. every $q \ge 1$ is feasible.
The drawback of that approach is that no intrinsic meaning is given to the equation,
that is a solution to \eqref{eq:introYoungBSDE} is only defined as the limit of smooth approximations.
In the current work we solve \eqref{eq:introYoungBSDE} directly via a fixed point argument.
The resulting object solves the integral equation, where the
integral with respect to $\rp$ is a pathwise Young integral.

In Section \ref{sec:mainResult} we state and prove our main result.
In Section \ref{sec:applications} we give an application to partial differential equations.
In Section \ref{sec:youngIntegration} we recall the notions of $p$-variation and Young integration.

\section{Main result}
\label{sec:mainResult}

We shall need the following spaces.
\begin{definition}
  For $p>2$ define $\bsdespace$ to be the space of
  adapted process $Y: \Omega \times [0,T] \to \R$ with\footnote{The space $C^{\pvar}([t,T], \R)$ and the norm $||\cdot||_{\pvar;[t,T]}$ are recalled
  in Section \ref{sec:youngIntegration}.}
  \begin{align*}
    \bsdenorm{ Y }
    :=
    \esssup_{t,\omega}\, \E_t\left[ ||Y||_{\pvar;[t,T]}^2 \right]^{1/2}
    +
    \esssup_{\omega} |Y_T| < \infty.
  \end{align*}

  Denote by $\bsdespacez$ the space of all
  progressively measurable $Z: \Omega \times [0,T] \to \R^{\dbm}$ with
  \begin{align*}
    \bsdenormz{ Z }
    :=
    \esssup_{t,\omega}\, \E_t\left[ \int_t^T |Z_r|^2 dr \right] < \infty.
  \end{align*}
\end{definition}

\begin{theorem}
\label{thm-mainResult}
  Let $T > 0$, $\xi \in L^\infty(\F_T)$.
  Let $q \in [1,2)$ and $\rp \in C^{0,\qvar}([0,T], \R^\drp)$.
  %
  Assume $f: \Omega \times [0,T] \times \R \times \R^\dbm \to \R$,
  satisfies for some $C_f > 0$, $\P-\text{a.s.}$,
  \begin{align*}
    \sup_{t \in [0,T]} |f(t,0,0)| &< C_f \\
    |f(t,y,z) - f(t,y',z')| &\le C_f \left( |y-y'| + |z-z'| \right).
  \end{align*}
  Let $g_1, \dots, g_\drp \in C^2_b(\R)$.
  Let $p > 2$ such that $1/p + 1/q > 1$.
  \begin{enumerate}[label=(\roman*)]
    \item There exists a unique
    $Y \in \bsdespace, Z \in \bsdespacez$ such that
    \begin{align}
    \label{BSDE}
      Y_t
      =
      \xi + \int_t^T f(r,Y_r,Z_r) dr + \int_t^T g(Y_r) d\rp_r - \int_t^T Z_r dW_r,
    \end{align}
    where the $d\rp$ integral is a well-defined (pathwise) Young integral.

    \item 
      If, for $i=1,2$,
      $$
      Y^i_t = \xi_i + \int_t^T f_i(s, Y^i_s, Z^i_s) ds + \int_t^T g( Y_s) d\eta_s - \int_t^T Z^i_sdW_s,
      $$
      and $\xi_1 \le \xi_2$, $f_1 \le f_2$.  Then $Y^1 \le Y^2$.

    \item 
      The solution mapping
      \begin{align*}
        L^\infty(\F_T) \times C^\qvar([0,T], \R^\drp) &\to \bsdespace \times \bsdespacez \\
        (\xi,\rp) &\mapsto (Y,Z),
      \end{align*}
      is locally uniformly continuous.

    \item
      Fixing $f, g$ there exists for every $M > 0$
      a $C(M) > 0$ such that for $\xi,\xi' \in \mathcal{F}_T$ with
      $||\xi||_\infty, ||\xi'||_\infty, ||\rp||_{\qvar;[0,T]} < M$
      we have for the corresponding solutions $(Y,Z), (Y',Z')$
      \begin{align*}
        |Y_0 - Y'_0| \le C(M) \E[ |\xi - \xi'|^2 ]^{1/2}.
      \end{align*}
  \end{enumerate}
\end{theorem}
\begin{remark}
  The refined continuity statement in (iv) will be imporant for our application to rough PDEs in Section \ref{sec:applications}.
\end{remark}
%
%
\begin{proof}
For $R>0$ define
\begin{align*}
  B(R) := \{ (Y,Z) : \bsdenorm{Y} < R, ||Z||_{\bsdespacez} < R \}.
\end{align*}

For $Y \in \bsdespace, Z \in \bsdespacez$ define $\Phi(Y,Z) := (\tilde Y, \tilde Z)$,
where
\begin{align*}
  \tilde Y_t
  =
  \xi + \int_t^T f(r,Y_r,Z_r) dr + \int_t^T g(Y_r) d\rp_r - \int_t^T \tilde Z_r dW_r.
\end{align*}
This is well-defined as usual in the BSDE literature
(see for example \cite{bib:pardouxPeng}),
by
setting
\begin{align*}
  \tilde Y_t := \E[ \xi + \int_t^T f(r,Y_r,Z_r) dr + \int_t^T g(Y_r) d\rp_r | \F_t ],
\end{align*}
and letting $\tilde Z$ be the integrand in the It\^o representation of the martingale
\begin{align*}
 \tilde Y_t + \int_0^t f(r,Y_r,Z_r) dr + \int_0^t g(Y_r) d\rp_r
\end{align*}
In what follows $A \lesssim B$ means there exists a constant $C > 0$
that is independent of $\rp, \xi$ such that $A \le C B$.
The constant is bounded for $||g||_{C^2_b}, C_f$ bounded.

\textbf{Unique existence on small interval}\\
We first show that for $T$ small enough, $\Phi$ leaves a ball invariant,
i.e.  for $T$ small enough, $R$ large enough
\begin{align*}
  \Phi( B(R) ) \subset B(R).
\end{align*}
  Let $(\tilde Y, \tilde Z) = \Phi(Y, Z)$, then
  \begin{align}
    \label{eq:intF}
    ||\int_t^\cdot f(r,Y_r,Z_r) dr||_{\pvar;[t,T]}
    &\le
    ||\int_t^\cdot f(r,Y_r,Z_r) dr||_{\onevar;[t,T]} \notag \\
    &\le
    \int_t^T |f(r,Y_r,Z_r)| dr \notag \\
    &\lesssim
    \int_t^T |f(r,0,0)| dr
    +
    T ||Y||_{\infty;[t,T]}
    +
    \int_t^T |Z_r| dr \notag \\
    &\lesssim
    T 
    +
    T ||Y||_{\pvar;[t,T]}
    +
    T |Y_T|
    +
    \int_t^T |Z_r| dr.
    %
  \end{align}
  Using the Young estimate (Theorem \ref{thm:youngIntegration} in the Appendix) we estimate
  \begin{align}
    \label{eq:intG}
    ||\int_t^\cdot g(Y_r) d\rp_r ||_{\pvar;[t,T]}
    &\le
    ||\int_t^\cdot g(Y_r) d\rp_r ||_{\qvar;[t,T]} \notag \\
    &\lesssim
    \left( 1 + ||Y||_{\pvar;[t,T]} \right) ||\rp||_{\qvar;[t,T]}.
  \end{align}

  The Burkholder-Davis-Gundy inequality for $p$-variation (\cite[Theorem 14.12]{bib:frizVictoir}) gives
  \begin{align}
    \label{eq:BDG}
    \E_t[ ||\int_t^\cdot \tilde Z_r dW_r||_{\pvar;[t,T]}^2 ] \lesssim \E_t[ \int_t^T |\tilde Z_r|^2 dr ].
  \end{align}

  Now the $d\rp$ integral satisfies the usual product rule,
  so together with It\=o's formula we get
  \begin{align*}
    \tilde Y_t^2 =
    \xi^2
    +
    2 \int_t^T f(r,Y_r,Z_r) \tilde Y_r dr
    +
    2 \int_t^T g(Y_r) \tilde Y_r d\rp_r
    -
    \int_t^T 2 \tilde Y_r \tilde Z_s dW_r
    -
    \int_t^T |\tilde Z_r|^2 dr.
  \end{align*}
  By Lemma \ref{lem:pvarProduct} (again in Appendix below) 
  \begin{align*}
    ||g(Y) \tilde Y||_{\pvar;[t,T]}
    &\le
    ||g||_\infty
    ||\tilde Y||_{\pvar;[t,T]}
    +
    ||g(Y)||_{\pvar;[t,T]}
    ||\tilde Y||_{\infty;[t,T]} \\
    &\le
    ||g||_\infty
    ||\tilde Y||_{\pvar;[t,T]}
    +
    ||Dg||_\infty ||Y||_{\pvar;[t,T]}
    \left(
      ||\tilde Y||_{\pvar;[t,T]}
      +
      |Y_T|
    \right)\\
    &\lesssim
    ||\tilde Y||_{\pvar;[t,T]}
    +
    ||\tilde Y||_{\pvar;[t,T]}^2
    +
    R^2.
  \end{align*}
  Taking conditional expectation we get
  \begin{align}
    \label{eq:taking}
    \E_t\left[ \tilde Y_t^2 \right]
    +
    \E_t\left[ \int_t^T |\tilde Z_s|^2 ds \right]
    &\lesssim
    \E_t[ \xi^2 ]
    +
    \E_t\left[ \int_t^T \left( |f(r,0,0)| + |Y_r| + |Z_r| \right) |\tilde Y_r| dr \right] \notag \\
    &\qquad
    +
    ||\rp||_\qvar \left( 1 + \E_t\left[ ||\tilde Y||_\pvar + ||\tilde Y||_\pvar^2 \right] + R^2 \right).
    %
  \end{align}

  Now
  \begin{align*}
    &\E_t\left[ \int_t^T \left( |f(r,0,0)| + |Y_r| + |Z_r| \right) |\tilde Y_r| dr \right] \\
    &\lesssim
    \E_t\left[ \int_t^T |f(r,0,0)|^2 + |Y_r|^2 +  |Z_r|^2 + |\tilde Y_r|^2 dr \right] \\
    &\lesssim
    \E_t\left[ \int_t^T |f(r,0,0)|^2  dr \right]
    +
    T \E_t\left[ ||Y||_\infty^2 \right]
    +
    \E_t\left[ \int_t^T |Z_r|^2 dr \right]
    +
    T \E_t\left[ ||\tilde Y||_\infty^2 \right] \\
    &\lesssim
    1
    +
    T \E_t\left[ ||Y||_\pvar^2 + |Y_T|^2 \right]
    +
    R
    +
    T \E_t\left[ ||\tilde Y||_\pvar^2 + |\tilde Y_T|^2 \right] \\
    &\lesssim
    1
    +
    T R^2
    +
    T\E_t\left[ \xi^2 \right]
    +
    R
    +
    T \E_t\left[ ||\tilde Y||_\pvar^2 \right]
    +
    T \E_t\left[ \xi^2 \right].
  \end{align*}
  We estimate trivially
  \begin{align*}
    \E_t\left[ ||\tilde Y||_{\pvar;[t,T]}^2 \right]
    &\lesssim
    \E_t\left[ ||\int_t^\cdot f(r,Y_r,Z_r) dr||_{\pvar;[t,T]}^2 \right]
    +
    \E_t\left[ ||\int_t^\cdot g(Y_r) d\rp_r ||_{\pvar;[t,T]}^2 \right] \\
    &\quad
    +
    \E_t\left[ ||\int_t^\cdot \tilde Z_r dW_r||_{\pvar;[t,T]}^2 \right],
  \end{align*}
  which we can bound, using \eqref{eq:intF}, \eqref{eq:intG} and \eqref{eq:BDG}, by a constant times
  \begin{align*}
    &T^2
    +
    T^2 \E_t\left[ ||Y||_{\pvar;[t,T]}^2 \right]
    +
    T^2 \E_t\left[ \xi^2 \right]
    +
    \left( 1 + T \right) \E\left[ \int_t^T |\tilde Z_r|^2 dr \right] \\
    &\qquad
    +
    \left( 1 +
    \E_t\left[ ||Y||_{\pvar;[t,T]}^2 \right]
    \right) ||\rp||_{\qvar;[t,T]}^2 \\
    &\lesssim
    T^2
    +
    T^2 R^2
    +
    T^2 \E_t\left[ \xi^2 \right]
    +
    \left( 1 + T \right) \E\left[ \int_t^T |\tilde Z_r|^2 dr \right] \\
    &\qquad
    +
    \left( 1 +
    R^2
    \right) ||\rp||_{\qvar;[t,T]}^2.
  \end{align*}
  Combining with \eqref{eq:taking}, we get
  \begin{align*}
    &\E_t\left[ \tilde Y_t^2 \right]
    +
    \E_t\left[ \int_t^T |\tilde Z_s|^2 ds \right] \\
    &\lesssim
    \left( 1 + T+ T^2 \right) \E_t[ \xi^2 ]
    +
    1
    +
    T R^2
    +
    R
    +
    T \E_t\left[ ||\tilde Y||_\pvar^2 \right] \\
    &\quad
    +
    ||\rp||_\qvar \left( 1 + \E_t\left[ ||\tilde Y||_\pvar + ||\tilde Y||_\pvar^2 \right] + R + R^2 \right) \\
    &\lesssim
    \left( 1 + T+ T^2 \right) \E_t[ \xi^2 ]
    +
    1
    +
    T R^2
    +
    R \\
    &\quad
    +
    T
    \left\{
      T^2
      +
      T^2 R^2
      +
      T^2 \E_t\left[ \xi^2 \right]
      +
      \left( 1 + T \right) \E\left[ \int_t^T |\tilde Z_r|^2 dr \right]
      +
      \left( 1 + R^2 \right) ||\rp||_{\qvar;[t,T]}^2
    \right\} \\
    &\quad
    +
    ||\rp||_\qvar \\
    &\quad
      \times \Biggl( 1
              + 
    \Bigl\{
    T
    +
    T R
    +
    T \E_t\left[ \xi^2 \right]^{1/2}
    +
    T^{1/2} \E\left[ \int_t^T |\tilde Z_r|^2 dr \right]^{1/2}
    +
    \left( 1 + R \right) ||\rp||_{\qvar;[t,T]} \\
    &\qquad
    +
      T^2
      +
      T^2 R^2
      +
      T^2 \E_t\left[ \xi^2 \right]
      +
      T \E\left[ \int_t^T |\tilde Z_r|^2 dr \right]
      +
      \left( 1 + R^2 \right) ||\rp||_{\qvar;[t,T]}^2
    \Bigr\}
    + R + R^2 \Biggr)
  \end{align*}
  Using $|a| \le 1 + |a|^2$ and picking $T > 0$ such that $T + T^2 \le 1/2$ we get,
  \begin{align*}
    \E_t\left[ \int_t^T |\tilde Z_s|^2 ds \right]
    \le
    c \left( 1 + F(T) \left( R + R^2 \right) \right),
  \end{align*}
  with $F(T) \to 0$, as $T \to 0$ (here we use that $||\rp||_{\qvar;[0,T]} \to 0$ for $T \to 0$).

  Then
  \begin{align*}
    \E_t\left[ ||\tilde Y||_{\pvar;[t,T]}^2 \right]^{1/2}
    \lesssim
    T
    +
    T R
    +
    T \E_t\left[ \xi^2 \right]^{1/2}
    +
    T^{1/2} \left( 1 + F(T) \left( R + R^2 \right) \right)
    +
    \left( 1 + R \right) ||\rp||_{\qvar;[t,T]},
  \end{align*}
  which can be made smaller than $R/2$ by picking first $R$ large and then $T$ small.
  So indeed the ball stays invariant.

  We now show that for $T$ small enough, $\Phi$ is a contraction on $B(R)$.
  So let $(Y,Z), (Y',Z') \in B(R)$ be given.
  Note that, since $Y_T = Y'_T$,
  \begin{align*}
    |Y_t - Y'_t|
    &=
    \E_t[ |Y_t - Y'_t| ] \\
    &\le
    \E_t[ ||Y - Y'||_{\pvar;[t,T]} ] \\
    &\le
    \E_t[ ||Y - Y'||_{\pvar;[t,T]}^2 ]^{1/2}.
  \end{align*}
  So
  \begin{align*}
    \esssup_{\omega} ||Y(\omega) - Y'(\omega)||_\infty
    \le
    \bsdenorm{ Y - Y' }.
  \end{align*}
  Let $(\tilde Y, \tilde Z) = \Phi(Y,Z), (\tilde Y', \tilde Z') = \Phi(Y',Z')$.
  Using the Young estimate (Theorem \ref{thm:youngIntegration}) and Lemma \ref{lem:differencePvar} (in Appendix below) we have
  for some constant $c$ that can change from line to line
  \begin{align*}
    ||\tilde Y - \tilde Y'||_{\pvar;[t,T]}
    &\le
    c T ||Y-Y'||_{\pvar;[t,T]}
    +
    c \int_t^T |Z_r-Z_r'| dr
    +
    ||Y - Y'||_{\pvar;[t,T]} ||\rp||_\qvar \\
    &\qquad
    +
    c \left( 1 + ||Y||_{\pvar;[t,T]} \right)
    || Y -  Y'||_\infty ||\rp||_\qvar
    +
    ||M - M'||_{\pvar;[t,T]},
  \end{align*}
  where $M = \int \tilde Z dW, M' = \int \tilde Z' dW$.
  Hence
  \begin{align*}
    \E_t[ ||\tilde Y - \tilde Y'||_{\pvar;[t,T]}^2 ]^{1/2}
    &\le
    c \E_t[ ||Y - Y'||_{\pvar;[t,T]} ]^{1/2} \left( T + ||\rp||_\qvar \right)
    +
    c T^{1/2} \E_t[ \int_t^T |Z_r-Z_r'|^2 dr ]^{1/2} \\
    &\quad
    +
    c \left( \E_t[ || Y||_{\pvar;[t,T]}^2 ]^{1/2} \right)
    \sup_{\omega} ||Y(\omega) -  Y'(\omega)||_\infty ||\rp||_\qvar \\
    &\qquad
    +
    c \E_t[ ||M - M'||_{\pvar;[t,T]}^2 ]^{1/2} \\
    &\le
    c T^{1/2} \E_t[ \int_t^T |Z_r-Z_r'|^2 dr ]^{1/2}
    +
    c \left( T + ||\rp||_\qvar \right) \bsdenorm{ Y - Y' } \\
    &\quad
    +
    c \left( 1 + R \right) ||\rp||_\qvar \bsdenorm{ Y - Y' }
    + 
    \E_t[ \int_t^T (\tilde Z_s - \tilde Z'_s)^2 ds ]^{1/2}.
  \end{align*}

  So for $T$ small enough
  \begin{align*}
    \bsdenorm{ \tilde Y - \tilde Y' }
    &\le
    \frac{1}{4} \left[ \bsdenorm{ Y - Y' } + \bsdenormz{ Z - Z' } \right]
    +
    \bsdenormz{ \tilde Z - \tilde Z' }.
  \end{align*}

  On the other hand
  \begin{align*}
    (\tilde Y_t - \tilde Y'_t)^2
    &=
    2 \int_t^T \left[ (f(Y_s,Z_s) - f(Y'_s,Z'_s)) (Y_s - Y'_s) \right] ds
    +
    2 \int_t^T \left[ (g(Y_s) - g(Y'_s)) (\tilde Y_s - \tilde Y'_s) \right] d\rp_s \\
    &\quad
    -
    2 \int_t^T \left[ (\tilde Y_s - Y'_s) (Z_s - Z'_s) \right] dB_s
    -
    \int_t^T |\tilde Z_s - \tilde Z_s'|^2 ds.
  \end{align*}

  Note that
  \begin{align*}
    &|| \left( g(Y) - g(Y') \right) \left( Y - Y' \right) ||_{\pvar;[t,T]} \\
    &\quad\lesssim
    ||g(Y) - g(Y')||_\infty ||Y-Y'||_{\pvar;[t,T]}
    +
    ||g(Y) - g(Y')||_{\pvar;[t,T]} ||Y-Y'||_\infty \\
    &\quad\lesssim
    ||Y-Y'||_{\pvar;[t,T]}^2
    +
    ( 1 + ||Y||_{\pvar;[t,T]} ) ||Y - Y'||_{\pvar;[t,T]}^2 \\
    &\quad\lesssim
    ( 1 + R ) ||Y-Y'||_{\pvar;[t,T]}^2.
  \end{align*}
  Hence the Young integral is bounded by a constant times $||\rp||_{\qvar;[t,T]} (1 + R) ||Y-Y'||_{\pvar;[t,T]}^2$.
  %
  We estimate the Lebesgue integral as
  \begin{align*}
    |\int_t^T \left[ (f(Y_s,Z_s) - f(Y'_s,Z'_s)) (Y_s - Y'_s) \right] ds|
    &\lesssim
    \int_t^T \left( |Y_s - Y'_s| + |Z_s - Z_s'| \right) |Y_s - Y'_s| ds \\
    &\lesssim
    T \left( 1 + \frac{1}{\lambda} \right) ||Y-Y||_\infty
    +
    \lambda \int_t^T |Z_s - Z_s'|^2 ds.
  \end{align*}

  So, after taking conditional expectation,
  \begin{align*}
    \E_t[ \int_t^T |\tilde Z_s - \tilde Z_s'|^2 ds ]^{1/2}
    &\lesssim
    T^{1/2} \left( 1 + \frac{1}{\lambda} \right)^{1/2} \E_t[ ||Y-Y'||_{{\pvar;[t,T]}}^2 ]^{1/2}
    +
    \lambda \E_t[ \int_t^T |Z_r - Z_r'|^2 dr ]^{1/2} \\
    &\qquad
    +
    ||\rp||_\qvar^{1/2} (1 + R)^{1/2} \E_t[ ||\tilde Y - \tilde Y'||_{\pvar;[t,T]}^2 ]^{1/2}.
  \end{align*}
  That is
  \begin{align*}
    \bsdenormz{\tilde Z - \tilde Z'}
    &\le
    T^{1/2} \left( 1 + \frac{1}{\lambda} \right)^{1/2} \bsdenorm{Y - Y'}
    +
    \lambda \bsdenormz{Z-Z'} \\
    &\qquad
    +
    ||\rp||_q (1 + R)^{1/2} \bsdenorm{Y - Y'}
  \end{align*}

  Picking $\lambda$ small, then $T$ small, we get
  \begin{align*}
    \bsdenormz{\tilde Z - \tilde Z'}
    \le
    \frac{1}{4} \bsdenorm{ Y - Y' }
    +
    \frac{1}{4} \bsdenormz{ Z - Z' }
  \end{align*}
  
  Define the modified norm
  \newcommand{\bsdenormjoint}[1]{\|\|#1\|\|}
  \begin{align*}
    \bsdenormjoint{ Y, Z }
    :=
    \bsdenorm{ Y }
    +
    2 \bsdenormz{ Z }.
  \end{align*}

  Then
  \begin{align*}
    &\bsdenormjoint{ \tilde Y - \tilde Y', \tilde Z - \tilde Z' } \\
    &\quad\le
    \frac{1}{4} \left[ \bsdenorm{ Y - Y' } + \bsdenormz{ Z - Z' } \right]
    +
    \bsdenormz{ \tilde Z - \tilde Z' }
    +
    \frac{1}{2} \bsdenorm{Y - Y'}
    +
    \frac{1}{2} \bsdenorm{Z - Z'} \\
    &\quad=
    \frac{3}{4} \bsdenorm{Y - Y'}
    +
    \frac{7}{4} \bsdenormz{Z - Z'} \\
    &\le
    \frac{7}{8} \bsdenormjoint{ Y - Y', Z - Z' }.
  \end{align*}
  We hence have a contraction
  and thereby existence of a unique solution on small enough time intervals.

  \textbf{Continuity on small time interval}\\
  This follows from virtually the same argument as the contraction mapping argument.

  \textbf{Comparison on small time interval}\\
  Let $C_B > 0$ be given,
  and pick $T = T(C_B)$ so small that the BSDE is well-posed
  for any $f,g$ with $||g||_{C_b^2}, C_f < C_B$ and
  any
  $\rp \in C^\qvar, \xi \in \F_T$ 
  with $||\rp||_{\qvar;[0,T]}, ||\xi||_\infty < C_B$.
  
  Let $\xi_1, \xi_2 \in \F_T$ be given with $||\xi_1||_\infty < C_B$
  and
  $\rp \in C^\qvar$  with
  $||\rp||_{\qvar;[0,T]} < C_B$.
  Let $\rp^n$ be a sequence of smooth paths approximating $\rp$ in $q$-variation norm, with $||\rp^n||_{\qvar;[0,T]} < C_B$ for all $n \ge 1$.

  Denote $Y_1^n$ resp $Y_2^n$ be the classical BSDE solution with driving path $\rp^n$
  and data $(\xi_1, f_1, g)$ resp. $(\xi_2, f_2, g)$.
  Then by standard comparison theorem
  (for example, see \cite{bib:ElKarouiPengQuenz})
  \begin{align*}
    Y_1^n \le Y_2^n.
  \end{align*}
  By continuity 
  we know that
  \begin{align*}
    \bsdenorm{Y_1^n - Y_1}
    +
    \bsdenorm{Y_2^n - Y_2} \to 0.
  \end{align*}
  In particular, almost surely,
  \begin{align*}
    ||Y_1^n - Y_1||_\infty
    +
    ||Y_2^n - Y_2||_\infty
    \to 0.
  \end{align*}
  Hence $Y_1 \le Y_2$.

  \textbf{Unique existence on arbitrary time interval}\\
  We show existence for arbitrary $T > 0$. Denote
  $$
  \overline \xi := \esssup_\omega \xi, ~~ \underline \xi := \essinf_\omega \xi,~~ \overline f := \esssup_\omega f,~~ \underline f := \essinf_\omega f.
  $$
  By assumption
  $$
    |\overline \xi|+ |\underline \xi| + \int_0^T [|\overline f|^2+|\underline f|^2](t,0,0) dt<\infty.
  $$ 
  Consider the following Young ODEs:
  \begin{align*}
  \overline Y_t = \overline\xi + \int_t^T \overline f(s, \overline Y_s, 0) + \int_t^T g( \overline Y_s) d\eta_s;\\
  \underline Y_t = \underline\xi + \int_t^T \underline f(s, \overline Y_s, 0) + \int_t^T g( \underline Y_s) d\eta_s.
  \end{align*}
  Note that $(\overline Y, 0)$ and $(\underline Y, 0)$ solve the following BSDEs respectively:
   \begin{align*}
  \overline Y_t = \overline\xi + \int_t^T \overline f(s, \overline Y_s, \overline Z_s) + \int_t^T g( \overline Y_s) d\eta_s - \int_t^T \overline Z_s dW_s;\\
  \underline Y_t = \underline\xi + \int_t^T \underline f(s, \underline Y_s, \underline Z_s) + \int_t^T g( \underline Y_s) d\eta_s -  \int_t^T \underline Z_s dW_s.
  \end{align*}

  Choose $\delta$ such that the BSDE (\ref{BSDE}) is
  wellposed on a time interval of length $\delta$ whenever the terminal condition is bounded by $\|\overline Y\|_\infty\vee \|\underline Y\|_\infty$.
  Let $\pi$: $0=t_0<\cdots<t_n=T$ be a partition
  such that $t_{i+1}-t_i\le \delta$ for all $i$.  First, by the preceeding arguments,
  BSDE (\ref{BSDE}) on $[t_{n-1}, t_n]$ with terminal
  condition $\xi$ is wellposed and we denote the solution by $(Y^n, Z^n)$. By
  comparison
  we have $\underline Y_{t_{n-1}} \le Y^n_{t_{n-1}} \le \overline Y_{t_{n-1}}$.
  We can hence start again the BSDE from $Y_{t_{n-1}}$ at time $t_{n-1}$
  and solve back to time $t_{n-2}$.
  Repeating the arguments backwardly we obtain the existence of a (unique)
  solution on $[0, T]$.

  \textbf{Continuity}\\
  Using the previous step we can use the continuity result on small intervals
  to get continuity of the solution map on arbitrary intervals.

  We finish by showing the second continuity statement.
  Since the $d\rp$-term is more difficult then the $dt$-term we will assume $f\equiv 0$
  for ease of presentation.
  %
  First note that since the $||\xi||_\infty, ||\xi'||_\infty < M$,
  the local uniform continuity of the solution map in Theorem \ref{thm-mainResult} we get
  \begin{align*}
    \bsdenorm{Y^n} \le C_0(M).
  \end{align*}
  Let 
  \begin{align*}
    \alpha_r := \int_0^1 \partial_y g( \theta Y_r + (1-\theta) Y'_r ) d\theta.
  \end{align*}

  Note that 
  \begin{align*}
    ||\alpha||_{\pvar;[t,T]}
    \le
    C_1(M)
    \Big(
    ||Y||_{\pvar;[t,T]}
    +
    ||Y'||_{\pvar;[t,T]}
    \Big)
  \end{align*}
  So that
  \begin{align*}
    \E_t[ ||\alpha||_{\pvar;[t,T]} ]
    \le
    C_2(M),
  \end{align*}
  for some constant $C_2(M)$.
  Let $\Delta Y := Y - Y'$.
  Then (almost surely)
  \begin{align*}
    ||\Delta Y||_{\infty;[t,T]}
    &\le
    ||Y||_{\infty;[t,T]}
    +
    ||Y'||_{\infty;[t,T]} \\
    &\le
    \bsdenorm{Y}
    +
    \bsdenorm{Y'} \\
    &\le
    2 C_0.
  \end{align*}

  Now
  \begin{align*}
    d \Delta Y_t
    =
    -
    \alpha_t \Delta Y_t d\rp_t
    +
    \Delta Z_t dW_t,
  \end{align*}

  By Ito's formula, together with the classical product rule for the $d\rp$-term, we get
  \begin{align*}
    d[ \exp( \int_0^t \alpha_r  d\rp_r ) \Delta Y_t ]
    =
    \exp( \int_0^t \alpha_r d\rp_r ) \Delta Z_t dW_t,
  \end{align*}
  so that if the latter is an honest martingale we get
  \begin{align*}
    |\Delta Y_0|
    =
    |\E[ \exp( \int_0^T \alpha_r d\rp_r ) \Delta Y_T ]|
    \le
    \E[ \exp( 2 \int_0^T \alpha_r d\rp_r ) ]^{1/2} \E[ (\Delta Y_T)^2 ]^{1/2}.
  \end{align*}

  Let us calculate the conditional moments of $\Gamma_t := \int_t^T \alpha_r d\rp_r$.
  First
  \begin{align*}
    \E_t\left[ ||\Gamma||_{\qvar;[t,T]} \right]
    &\le
    c_{Young} ||\rp||_{\qvar;[t,T]} \E_t\left[ ||\alpha||_{\pvar;[t,T} \right] \\
    &\le
    c_{Young} ||\rp||_{q;[t,T]}
    C_2.
  \end{align*}

  Further, by the product rule,
  \begin{align*}
    (\Gamma_t)^{m+1} = (m+1) \int_t^T \Gamma^m_r \alpha_r d\rp_r,
  \end{align*}
  so that
  \begin{align*}
    &\E_t[ || (\Gamma)^{m+1} ||_{\qvar;[t,T]} ] \\
    &\le
    c_{Young}
    (m+1)
    ||\rp||_{\qvar;[t,T]}
    \E_t\Big[
      ||\Gamma^m||_{\pvar;[t,T]} ||\alpha||_{\infty;[t,T]}
      +
      ||\Gamma^m||_{\infty;[t,T]} ||\alpha||_{\pvar;[t,T]}
      \Big] \\
    &\le
    c_{Young}
    (m+1)
    ||\rp||_{\qvar;[t,T]}
    \Big(
      \E_t[ ||\Gamma^m||_{\pvar;[t,T]} ] ||g'||_\infty
      +
      \sup_{s \in [t,T]} \E_s[ ||\Gamma^m||_{\pvar;[s,T]} ]
      \E_t[ ||\alpha||_{\pvar;[t,T]} ]
    \Big) \\
    &\le
    c_{Young}
    (m+1)
    ||\rp||_{\qvar;[t,T]}
    \Big(
      \E_t[ ||\Gamma^m||_{\pvar;[t,T]} ] ||g'||_\infty
      +
      \sup_{s \in [t,T]} \E_s[ ||\Gamma^m||_{\pvar;[s,T]} ]
      C_2
    \Big)
  \end{align*}

  Iterating, we get that for some $C_3(M) > 0$
  \begin{align*}
    \sup_{t\le T} \E_t[ ||(\Gamma)^m||_{\qvar;[t,T]} ]
    \le
    m! C_3(M)^m.
  \end{align*}
  
  In particular, for every $t\le T$
  \begin{align*}
    \E[ (\Gamma_t)^m ]
    \le
    m! C_3(M)^m.
  \end{align*}
  So there is $\varepsilon > 0$ such that
  \begin{align*}
    \E[ \exp( \varepsilon |\int_t^T \alpha_r \Delta Y_r d\rp_r|^2 ) ] < C_3(M)
  \end{align*}
  In particular for every $c \in \R$
  \begin{align*}
    \E[ \exp( c |\int_t^T \alpha_r \Delta Y_r d\rp_r| ) ] < C_4(c,M).
  \end{align*}

  So the statement follows with $C(M) = C_4(2,M)$ if
  \begin{align*}
    \int \exp( \int_0^t \alpha_r d\rp_r ) \Delta Z_t dW_t,
  \end{align*}
  is an honest martingale.
  But this follows from
  \begin{align*}
   \E\left[ 
     \left(
      \int_0^T \exp( 2 \int_0^t \alpha_r d\rp_r ) |\Delta Z_t|^2 dt
    \right)^{1/2} \right]
    &\le
   \E\left[
     \left(
     \sup_{t \le T} \exp( 2 \int_0^t \alpha_r d\rp_r ) 
      \int_0^T |\Delta Z_t|^2 dt
    \right)^{1/2} \right] \\
    &\le
    \E\left[ \sup_{t \le T} \exp( 4 \int_0^t \alpha_r d\rp_r ) \right]^{1/2}
    \E\left[ \int_0^T |\Delta Z_t|^2 dt \right] \\
    &< \infty.
  \end{align*}

  Here we used
  \begin{align*}
    \E\left[ \exp( \int_0^t \alpha_r \Delta Y_r d\rp_r ) \right]
    &=
    \E\left[
      \exp( \int_0^T \alpha_r \Delta Y_r d\rp_r )
      \exp( -\int_t^T \alpha_r \Delta Y_r d\rp_r )
    \right] \\
    &\le
    \E\left[ \exp( 2 |\int_0^T \alpha_r \Delta Y_r d\rp_r| ) \right]^{1/2}
    \E\left[
      \exp( 2 |\int_t^T \alpha_r \Delta Y_r d\rp_r| )
    \right]^{1/2} \\
    &< \infty.
  \end{align*}
\end{proof}


\section{Application to rough PDEs}
\label{sec:applications}

\newcommand{\BUC}{\operatorname{BUC}}
\newcommand{\BC}{\operatorname{BC}}
It is well-known that BSDEs provide a stochastic representation
for solutions to semi-linear parabolic partial differential equations (PDEs),
in what is sometimes called the ``nonlinear Feynman-Kac formula''
\cite{bib:pardouxPengPDEs}.
In this section we show how to use BSDEs with Young drift for
the stochastic representation for PDEs of the form
\begin{align*}
  \partial_t u = \frac{1}{2} \Tr[ \sigma(x) \sigma^T(x) D^2 u^n ] + b(x) \cdot Du + f(t,u,\sigma(x)^T Du) + g(u) \dot{\rp}_t.
\end{align*}
Here $\rp$ has finite $q$-variation, with $q \in [1,2)$
and the last term is hence not well-defined.
There are several approaches to make sense of such a ``rough'' PDE.
Here we shall define the solution as the limit of solutions to smooth approximations,
see Theorem \ref{thm:roughPDE} below.

For $D = [0,T] \times \R^\dode$ or $D = \R^\dode$
we shall need the space $\BC(D)$ (resp. $\BUC(D)$) of bounded continous (resp. uniformly continous functions) on $D$.

Let us recall the nonlinear Feyman-Kac formula for standard PDEs.
\begin{theorem}[{\cite[Section 4]{bib:pardouxPengPDEs}}]
  \label{thm:smoothPDE}
  Let $h \in \BUC(\R^\dode)$,
  $f(t,y,z): [0,T] \times \R \times \R^\dode \to \R$
  bounded and Lipschitz in $y,z$ uniformly in $t,x$,
  $\sigma: \R^\dode \to L(\R^\dbm, \R^\dode)$ Lipschitz,
  $b: \R^\dode \to \R^\dode$ Lipschitz
  and $g_1, \dots, g_\drp \in C^2_b(\R)$
  and let $\rp$ be a smooth path.
  For every $s \in [0,T], x \in \R^\dode$
  let $X^{s,x}$ be the solution to the SDE
  \begin{align*}
    dX^{s,x}_t = \sigma(X^{s,x}_t) dW_t + b(X^{s,x}_t) dt \qquad X^{s,x}_s = x
  \end{align*}
  and $Y^{s,x}$ the solution to the BSDE
  \begin{align*}
    dY^{s,x}_t = f(t, Y^{s,x}_t, Z^{s,x}_t) dt + g(Y^{s,x}_t) d\rp_t - Z^{s,x}_t dW_t \qquad Y^{s,x}_T = h(X^{s,x}_T).
  \end{align*}

  Then $u(t,x) := Y^{t,x}_t$ is the unique viscosity solution in $\BUC([0,T]\times\R^\dode)$ to the PDE
  \begin{align*}
    \partial_t u &= \frac{1}{2} \Tr[ \sigma(x) \sigma^T(x) D^2 u^n ] + b(x) \cdot Du + f(t,u,\sigma(x)^T Du) + g(u) \dot{\rp}_t \\
            u|_T &= h.
  \end{align*}
\end{theorem}

The following theorem extends this representation property to BSDEs with Young drift.
\begin{theorem}
  \label{thm:roughPDE}
  Let $\rp \in C^{0,\qvar}$, $q \in [1,2)$ and let $\rp^n$ smooth be given such that $\rp^n \to \rp$ in $C^{0,\qvar}$.
  Let $f(t,y,z): [0,T] \times \R \times \R^\dode \to \R$
  bounded and Lipschitz in $y,z$ uniformly in $t,x$,
  $\sigma: \R^\dode \to L(\R^\dbm, \R^\dode)$ Lipschitz,
  $b: \R^\dode \to \R^\dode$ Lipschitz
  and $g_1, \dots, g_\drp \in C^2_b(\R)$.%

  Let $u^n$ be the unique $\BUC([0,T]\times\R^\dode)$ viscosity solution to
  \begin{align*}
    \partial_t u^n = \frac{1}{2} \Tr[ \sigma(x) \sigma^T(x) D^2 u^n ] + b(x) \cdot Du^n + f(t,u, \sigma(x)^T Du) + g(u) \dot{\rp}^n_t.
  \end{align*}
  Then there exists $u \in \BC([0,T]\times\R^\dode)$ such that $u^n \to u$ locally uniformly
  and the limit does not depend on the approximating sequence.
  Formally, $u$ solves the PDE
  \begin{align*}
    \partial_t u = \frac{1}{2} \Tr[ \sigma(x) \sigma^T(x) D^2 u^n ] + b(x) \cdot Du + f(t,u,\sigma(x)^T Du) + g(u) \dot{\rp}_t.
  \end{align*}
  Moreover $u(t,x) = Y^{t,x}_t$, where
  $X^{s,x}$ is the solution to the SDE
  \begin{align*}
    dX^{s,x}_t = \sigma(X^{s,x}_t) dW_t + b(X^{s,x}_t) dt \qquad X^{s,x}_s = x
  \end{align*}
  and $Y^{s,x}$ the solution to the BSDE with Young drift
  \begin{align*}
    dY^{s,x}_t = ft, Y^{s,x}_t, Z^{s,x}_t) dt + g(Y^{s,x}_t) d\rp_t - Z^{s,x}_t dW_t \qquad Y^{s,x}_T = h(X^{s,x}_T).
  \end{align*}
\end{theorem}
\begin{proof}
  By Theorem \ref{thm:smoothPDE} we can write $u^n(t,x) = Y^{n,t,x}_t$ where
  \begin{align*}
    dX^{s,x}_t &= \sigma(X^{s,x}_t) dW_t + b(X^{s,x}_t) dt, \ X^{s,x}_s = x
  \end{align*}
  and $Y^{n,s,x}$ is the solution to the BSDE
  \begin{align*}
    dY^{n,s,x}_t &= f(Y^{n,s,x}_t, Z^{n,s,x}_t) dt + g(Y^{n,s,x}_t) d\rp^n_t - Z^{n,s,x}_t dW_t, \ Y^{n,s,x}_T = h(X^{s,x}_T).
  \end{align*}

  By Theorem \ref{thm-mainResult} we have
  that for fixed $s,x$, $Y^{n,s,x} \to Y^{s,x}$ in $\bsdespace$, where
  $Y^{s,x}$ solves the corresponding BSDE with Young drift.
  In particular $Y^{n,s,x}_s \to Y^{s,x}_s$, and hence we get pointwise convergence of $u^n$.
  We now show that $u^n$ is locally uniformly continuous in $(t,x)$ uniformly in $n$.
  By Theorem \ref{thm-mainResult} (iv), uniformly in $n$,
  \begin{align*}
    |Y^{n,s,x}_s - Y^{n,s,x'}_s|
    &\le
    C \E[ |h(X^{s,x}_T) - h(X^{s,x'}_T)|^2 ]^{1/2} \\
    &\le
    C ||Dh||_\infty
    \E[ |X^{s,x}_T - X^{s,x'}_T|^2 ]^{1/2} \\
    &\lesssim
    |x-x'|,
  \end{align*}
  where we used Lipschitzness of the map $\R^\dode \ni x \mapsto X^{s,x}_T \in L^2(\Omega)$, see for example
  \cite[Theorem 2.2]{bib:stroock}.
  Moreover, for any $\delta>0$ small,
  \begin{align*}
    Y^{n,s+\delta,x}_{s+\delta}
    -
    Y^{n,s,x}_{s}
    &=
    \E[
      Y^{n,s+\delta,x}_{s+\delta}
      -
      Y^{n,s,x}_{s+\delta} ]
    +
    \E[
      Y^{n,s,x}_{s+\delta}
      -
      Y^{n,s,x}_{s} ] \\
    &=
    \E[
      Y^{n,s+\delta,x}_{s+\delta}
      -
      Y^{n,s+\delta,X^{s,x}_{s+\delta}}_{s+\delta} ]
    +
    \E[
      Y^{n,s,x}_{s+\delta}
      -
      Y^{n,s,x}_{s} ] \\
    &\lesssim
    \E[ |x - X^{s,x}_{s+\delta}|^2 ]^{1/2}
    +
    \E[ \int_s^{s+\delta} f(Y^{n,s,x}_r, Z^{n,s,x}_r) dr + \int_s^{s+\delta} g(Y^{n,s,x}_r) d\rp^n_r ] \\
    &\lesssim
    \delta^{1/2}
    +
    \delta ||f||_\infty
    +
    ||\rp||_{\qvar;[s,s+\delta]}
    \E[ \left( 1 + ||Y^{n,s,x}||_{\pvar;[s,s+\delta]} \right) ] \\
    &\lesssim
    \delta^{1/2}
    +
    \delta ||f||_\infty
    +
    ||\rp||_{\qvar;[s,s+\delta]},
  \end{align*}
  where we used the uniform boundedness of $\bsdenorm{Y^n}$ in the last step (as in the proof of Theorem \ref{thm-mainResult}).

  It follows that $u^n$ is locally uniformly continuous in $(t,x)$ uniformly in $n$.
  Hence $u^n$ converges to $u$ locally uniformly.
\end{proof}

\begin{remark}
  In the vain of \cite{bib:diehlFrizStannat} one can also,
  under appropriate assumptions on the coefficients, verify that $u$ solves an integral equation.
\end{remark}


%

\section{Appendix - Young integration}
\label{sec:youngIntegration}

For $p\ge 1$, $V$ some Banach space, we denote by $C^\pvar = C^\pvar([0,T],V)$
the space of $V$-valued continuous paths $X$ with finite $p$-variation
\begin{align*}
  ||X||_\pvar
  := ||X||_{\pvar;[0,T]}
  := \left( \sup_\pi \sum_{[u,v] \in \pi} |X_{u,v}|^p \right)^{1/p}.
\end{align*}
Here the supremum runs over all partitions of the interval $[0,T]$
and $X_{u,v} := X_v - X_u$.

We shall also need the space $C^{0,\pvar} = C^{0,\pvar}([0,T], V)$,
defined as the closure of $C^\infty([0,T], V)$ under the norm $||\cdot||_\pvar$.
Obviously $C^{0,\pvar} \subset C^\pvar$, and the inclusion is strict \cite[Section 5.3.3]{bib:frizVictoir}.

The following basic estimates can be found in \cite[Chapter 5]{bib:frizVictoir}
\begin{align*}
  ||Y||_\infty &\le |Y_T| + ||Y||_\pvar, \qquad \forall p \ge 1 \\
  ||Y||_\pvar &\le ||Y||_\qvar, \qquad \forall 1 \le q \le p.
\end{align*}

The proof of the following result goes back
to \cite{bib:young}.
A short modern proof can be found in \cite[Chapter 4]{bib:frizHairer}.
In this statement and in what follows $a \lesssim b$,
means that there exists a constant $c > 0$, not depending on the
paths under considerations, such that $a \le c b$.
The constant $c$ can depend on the vector fields under considerations,
the dimension and the time horizion $T$, but is bounded for $T$ bounded.
\begin{theorem}[Young integration]
  \label{thm:youngIntegration}
  Let $X \in C^\pvar( [0,T], L(V,W) )$, $Y \in C^\qvar( [0,T], W )$
  with $1/p + 1/q > 1$.
  Then
  \begin{align*}
    \int_0^T X_s dY_s
    :=
    \lim_{|\pi| \to 0}
    \sum_{[u,v] \in \pi} X_u Y_{u,v}
  \end{align*}
  exists, where the limit is taken over partitions of $[0,T]$ with meshsize going to $0$.
  Moreover
  \begin{align*}
    || \int X_s dY_s ||_{\qvar;[0,T]}
    &\lesssim
    ( |X_0| + ||X||_{\pvar;[0,T]} ) ||Y||_{\qvar;[0,T]} \\
    &\lesssim
    ( |X_T| + ||X||_{\pvar;[0,T]} ) ||Y||_{\qvar;[0,T]}.
  \end{align*}
\end{theorem}

We also need
\begin{lemma}
  \label{lem:differencePvar}
  Let $p\ge 1$, $g \in C^2_b$, $a, a' \in C^\pvar$, then
  \begin{align*}
    ||g(a) - g(a')||_\pvar \le c ||a - a'||_\pvar + \left( ||a||_\pvar + ||a'||_\pvar \right) ||a-a'||_\infty.
  \end{align*}
\end{lemma}
\begin{proof}
  This follows from
  \begin{align*}
    &|
    g(a_t) - g(a'_t)
    -
    g(a_s) - g(a'_s)
    | \\
    &\quad=
    |
    \int_0^1 Dg( a'_t + \theta (a_t - a'_t) ) d\theta (a_t - a'_t)
    -
    \int_0^1 Dg( a'_s + \theta (a_s - a'_s) ) d\theta (a_s - a'_s)
    | \\
    &\quad\le
    |
    \int_0^1 Dg( a'_t + \theta (a_t - a'_t) ) 
            -   Dg( a'_s + \theta (a_s - a'_s) )
             d\theta (a_t - a'_t)| \\
    &\qquad
    +
    |\int_0^1 Dg( a'_s + \theta (a_s - a'_s) ) d\theta (a_t - a'_t) - (a_s - a'_s)| \\
    &\quad\le
    ||D^2 g||_\infty
    \left( |a'_t - a'_s| + |a_t - a_s| \right) |a_t - a'_t|
    +
    ||Dg||_\infty |(a_t - a'_t) - (a_s - a'_s)|.
  \end{align*}
\end{proof}

\begin{lemma}
  \label{lem:pvarProduct}
  Let $p\ge 1$ and $a, b \in C^\pvar$ then
  \begin{align*}
    ||a b||_\pvar
    \lesssim
    ||a||_\pvar ||b||_\infty
    +
    ||a||_\infty ||b||_\pvar
  \end{align*}
\end{lemma}
\begin{proof}
  This follows from
  \begin{align*}
    |a_t b_t - a_s b_s|
    \le
    |a_t - a_s| ||b||_\infty
    +
    ||a||_\infty |b_t - b_s|.
  \end{align*}
\end{proof}

\end{document}